\documentclass{amsart}
\usepackage{amsmath}
\usepackage{amstext}
\usepackage{amssymb}
\usepackage{amsthm}
\usepackage{amscd}
\swapnumbers
\theoremstyle{plain}
\newtheorem{thm}{Theorem}[section]
\newtheorem{lem}[thm]{Lemma}
\newtheorem{prop}[thm]{Proposition}
\newtheorem{cor}[thm]{Corollary}

\theoremstyle{definition}

\newtheorem{defs}[thm]{Definitions}

\newtheorem{exs}[thm]{Examples}
\newtheorem{qu}[thm]{Question}
\newtheorem{ntn}[thm]{Notation}

\theoremstyle{remark}

\newtheorem{rmk}[thm]{Remark}
\newtheorem{disc}[thm]{Discussion}

\DeclareMathOperator{\Id}{Id} 
 \DeclareMathOperator{\Hom}{Hom}
 
\DeclareMathOperator{\Ker}{Ker} \DeclareMathOperator{\Ima}{Im}

\DeclareMathOperator{\ann}{ann} 
\DeclareMathOperator{\Mod}{Mod} \DeclareMathOperator{\Endo}{End}
\DeclareMathOperator{\grann}{gr-ann}
\def\Z{\mathbb Z}
\def\N{\mathbb N}

\def\H{{\bf H}}
\def\M{{\bf M}}
\def\D{{\bf D}}

\def\fa{{\mathfrak{a}}}

\def\fb{{\mathfrak{b}}}
\def\fB{{\mathfrak{B}}}

\def\fm{{\mathfrak{m}}}

\def\fp{{\mathfrak{p}}}

\def\fq{{\mathfrak{q}}}

\def\nn{\relax\ifmmode{\mathbb N_{0}}\else$\mathbb N_{0}$\fi}
\def\lra{\longrightarrow}

\begin{document}

\title[RIGHT AND LEFT MODULES OVER THE FROBENIUS SKEW POLYNOMIAL RING]{RIGHT AND LEFT MODULES OVER THE
FROBENIUS SKEW POLYNOMIAL RING IN THE $F$-FINITE CASE}
\author{RODNEY Y. SHARP}
\address{Department of Pure Mathematics,
University of Sheffield, Hicks Building, Sheffield S3 7RH, United
Kingdom} \email{R.Y.Sharp@sheffield.ac.uk}
\author{YUJI YOSHINO}
\address{Department of Mathematics, Faculty of Science, Okayama University, Tsushima-Naka 3-1-1,
Okayama 700-8530, Japan} \email{yoshino@math.okayama-u.ac.jp}

\thanks{The first author was partially supported by the
Engineering and Physical Sciences Research Council of the United
Kingdom (Overseas Travel Grant Number EP/C538803/1), and also by the
Foundation for International Exchange Program of Okayama University.
The second author was partially supported by Japan Society for the
Promotion of Science (Grant-in-Aid (B) 21340008). }


\date{\today}



\begin{abstract}
The main purposes of this paper are to establish and exploit the
result that, over a complete (Noetherian) local ring $R$ of prime
characteristic for which the Frobenius homomorphism $f$ is finite,
the appropriate restrictions of the Matlis-duality functor provide
an equivalence between the category of left modules over the
Frobenius skew polynomial ring $R[x,f]$ that are Artinian as
$R$-modules and the category of right $R[x,f]$-modules that are
Noetherian as $R$-modules.
\end{abstract}

\maketitle


\setcounter{section}{-1}
\section{\sc Introduction}
\label{in}

Throughout the paper, $R$ will denote a commutative Noetherian ring
of prime characteristic $p$. We shall only assume that $R$ is local
when this is explicitly stated; then, the notation `$(R,\fm)$' will
denote that $\fm$ is the maximal ideal of $R$. We shall always
denote by $f:R\lra R$ the Frobenius homomorphism, for which $f(r) =
r^p$ for all $r \in R$. We shall work with the skew polynomial ring
$R[x,f]$ associated to $R$ and $f$ in the indeterminate $x$ over
$R$. Recall that $R[x,f]$ is, as a left $R$-module, freely generated
by $(x^i)_{i \in \nn}$ (we use $\N$ and $\nn$ to denote the set of
positive integers and the set of non-negative integers,
respectively), and so consists  of all polynomials $\sum_{i = 0}^n
r_i x^i$, where $n \in \nn$  and  $r_0,\ldots,r_n \in R$; however,
its multiplication is subject to the rule
$$
xr = f(r)x = r^px \quad \mbox{~for all~} r \in R\/.
$$
Note that $R[x,f]$ can be considered as a positively-graded ring
$R[x,f] = \bigoplus_{n=0}^{\infty} R[x,f]_n$, with $R[x,f]_n = Rx^n$
for all $n \in \nn$. The ring $R[x,f]$ will be referred to as the
{\em Frobenius skew polynomial ring over $R$.}

In the case when $(R,\fm)$ is local, several authors have used,
often as an aid to the study of tight closure, the natural Frobenius
action on the top local cohomology module $H^{\dim R}_{\fm}(R)$ of
$R$: see, for example, R. Fedder \cite{F87}, Fedder and K.-i.
Watanabe \cite{FW87}, K. E. Smith \cite{S94}, N. Hara and Watanabe
\cite{HW96} and F. Enescu \cite{Enesc03}, \cite{Enesc09}. The
natural Frobenius action provides the top local cohomology module of
$R$ with a natural structure as a left module over $R[x,f]$. The top
local cohomology module of $R$ is Artinian as $R$-module, and so the
papers cited above studied one example of a left $R[x,f]$-module
that is Artinian as $R$-module. In recent years there have been
studies of more general left $R[x,f]$-modules that are Artinian as
$R$-modules: see, for example, M. Katzman \cite{MK} and the first
author's \cite{ga}, \cite{gatcti} and \cite{Fpur} (the authors are
listed alphabetically).

On the other hand, the second author showed in \cite[Proposition
3.5]{Yoshi94} that, if $R$ is {\em $F$-finite\/}, that is, the
Frobenius map $f: R \lra R$ is a finite homomorphism, then each
non-zero injective $R$-module $I$ has a non-trivial structure as a
right $R[x,f]$-module. The main purpose of this paper is to build on
that work to show that, when $R$ is $F$-finite, whenever $M$ is a
left $R[x,f]$-module, then $\Hom_R(M,I)$ can be given a structure as
right $R[x,f]$-module that extends its $R$-module structure, and,
furthermore, whenever $N$ is a right $R[x,f]$-module, then
$\Hom_R(N,I)$ can be given a structure as left $R[x,f]$-module that
extends its $R$-module structure. Special attention is given to the
case where $(R,\fm)$ is local, complete and $F$-finite, and $I$ is
taken to be $E := E_R(R/\fm)$, the injective envelope of the simple
$R$-module. Classical Matlis duality yields that whenever $G$ is an
$R$-module that is Artinian (respectively Noetherian), then the
natural `evaluation' $R$-homomorphism $G \lra \Hom_R(\Hom_R(G,E),E)$
is an isomorphism, and the `Matlis dual' $\Hom_R(G,E)$ of $G$ is
Noetherian (respectively Artinian). Our results, when combined with
Matlis duality, lead to the conclusion that the appropriate
restrictions of the functor $\Hom_R(-, E)$ provide an equivalence
between the category of left $R[x,f]$-modules that are Artinian as
$R$-modules (and all $R[x,f]$-homomorphisms between them) and the
category of right $R[x,f]$-modules that are Noetherian as
$R$-modules (and all $R[x,f]$-homomorphisms between them).

We can then use this equivalence to translate (in this complete,
local, $F$-finite case) known results about left $R[x,f]$-modules
that are Artinian as $R$-modules into results about right
$R[x,f]$-modules that are Noetherian as $R$-modules. One example of
this concerns the Hartshorne--Speiser--Lyubeznik Theorem, which we
now recall.

\begin{thm} [G. Lyubeznik {\cite[Proposition 4.4]{Lyube97}}]
\label{hs.4} {\rm (Compare Hartshorne--Speiser \cite[Proposition
1.11]{HarSpe77}.)} Suppose that $(R,\fm)$ is local, and let $G$ be a
left $R[x,f]$-module that is Artinian as $R$-module. Then there
exists $e \in \nn$ with the following property: whenever $g \in G$
is such that $x^ng = 0$ for some $n \in \N$, then $x^eg =
0$.\end{thm}

Hartshorne and Speiser first proved this result in the case where
$R$ is local and contains its residue field which is perfect.
Lyubeznik applied his theory of $F$-modules to obtain the result
without restriction on the local ring $R$ of characteristic $p$.
There is a short proof of the Hartshorne--Speiser--Lyubeznik Theorem
in \cite{HSLonly}. It was shown in \cite[Corollary 1.8]{mmj} that
the result is still valid if the hypothesis that $R$ be local is
dropped.

The Hartshorne--Speiser--Lyubeznik Theorem has been used to
establish the existence of uniform test exponents for Frobenius
closures of parameter ideals in local rings in certain
circumstances. Let $\fa$ be an ideal of $R$; let $n \in \nn$. Recall
that the {\em $n$-th Frobenius power\/} $\fa^{[p^n]}$ of $\fa$ is
the ideal of $R$ generated by all $p^n$-th powers of elements of
$\fa$. The {\em Frobenius closure $\fa^F$} of $\fa$ is defined by
$$
\fa ^F := \big\{ r \in R \ | \ \mbox{there exists~} n\in\nn
\mbox{~such that~} r^{p^n} \in \fa^{[p^n]}\big\}.
$$
This is an ideal of $R$, and so is finitely generated; therefore
there exists a power $Q_0$ of $p$ such that $(\fa^F)^{[Q_0]} =
\fa^{[Q_0]}$, and we define $Q(\fa)$ to be the smallest power of $p$
with this property. In \cite[Theorem 2.5]{KS}, M. Katzman and Sharp
used the Hartshorne--Speiser--Lyubeznik Theorem to show that, when
$(R,\fm)$ is local and Cohen--Macaulay, the set
$$
\{Q(\fa) : \fa \mbox{~is an ideal generated by part of a system of
parameters of~}R\}
$$
is bounded; in \cite{HuKSY}, C. Huneke, Katzman, Sharp and Y. Yao
again used the Hartshorne--Speiser--Lyubeznik Theorem (and quite a
few other techniques) to establish the same conclusion in a
generalized Cohen--Macaulay local ring.

We are able to use our above-mentioned equivalence of categories
to prove the following result (as Theorem \ref{appl.1}),
in the case where $R$ is $F$-finite, local and complete.

\vspace{0.1in}

\noindent{\bf Theorem.} {\it
Assume that $(R,\fm)$ is $F$-finite, local and complete.
Let $N$ be a right $R[x,f]$-module that is Noetherian as $R$-module.
Then there exists $e \in \nn$ such that $Nx^e = Nx^{e+1}$.}

\vspace{0.1in}

This result can be viewed as a dual of the
Hartshorne--Speiser--Lyubeznik Theorem. A natural question is
whether this `dual Hartshorne--Speiser--Lyubeznik Theorem' is still
valid if all the hypotheses about $R$, except the one that it (is a
commutative Noetherian ring and) has characteristic $p$, are
dropped: we shall show, in the final section of the paper, that this
question has an affirmative answer.

Another useful result about left $R[x,f]$-modules that are Artinian
as $R$-modules concerns graded annihilators: the {\em graded
annihilator\/} of a (left or right) $R[x,f]$-module $T$ is the
largest graded two-sided ideal of $R[x,f]$ that annihilates $T$.

\begin{thm} [R. Y. Sharp {\cite[Corollary 3.11]{ga}}]
\label{in.2} Let $G$ be a left $R[x,f]$-module that is Artinian as
$R$-module. Suppose that $G$ is $x$-torsion-free, that is, $xg = 0$
for $g \in G$ implies that $g = 0$. Then there are only finitely
many graded annihilators of $R[x,f]$-submodules of $G$.
\end{thm}

The first author has been able to use this result to prove existence
theorems about tight closure test elements: see \cite[Theorem
4.16]{Fpur}.

We are able to use our above-mentioned equivalence of categories to
prove the following result (as Theorem \ref{appl.3}), in the case
where $R$ is $F$-finite, local and complete.

\vspace{0.1in}

\noindent{\bf Theorem.} {\it Assume that $(R,\fm)$ is $F$-finite,
local and complete. Let $M $ be a right $R[x,f]$-module that is
Noetherian as $R$-module. Suppose that $M$ is $x$-divisible, that is
$M = Mx$. Then there are only finitely many graded annihilators of
$R[x,f]$-homomorphic images of $M$.}

\vspace{0.1in}

Again, it is natural to ask whether this result is still valid if
all the hypotheses about $R$, except the one that it (is a
commutative Noetherian ring and) has characteristic $p$, are
dropped.  At the time of writing, we have not been able to answer to
this question.

\vspace{0.1in}

{\bf Note.}  Most of the research reported in this paper was carried
out during a visit by Sharp to the University of Okayama in March
2008.  After the paper had been accepted, it was pointed out to us
that some of its results have been independently obtained by M.
Blickle and G. Boeckle in their paper \cite{BB}. In detail, Theorem
1.20 below appears in \cite[Section 5.1]{BB}, and the result of
Theorem 3.4 below follows from \cite[Proposition 2.14]{BB} (which
Blickle and Boeckle prove via an argument of O. Gabber from
\cite[Section 13]{Gabber}).

\section{\sc Right and left modules over the Frobenius skew polynomial ring}
\label{nt}

The notation and terminology used in the Introduction will be used throughout the paper.

First of all, let us recall some of the basic facts about bimodules,
which we shall use in the rest of the paper.  See, for example,
Rotman \cite[Lemma 8.80, Theorem 8.99]{Rotman02}.

\begin{rmk}\label{bimod}
Let $A$, $B$, $C$ and $D$ be commutative rings.

\begin{itemize}
\item[{(\rm i)}]
An Abelian group $M$ is an $(A, B)$-bimodule if $M$ is a left
$A$-module, a right $B$-module and the two actions of the rings are
related by the following rule:
$$
(am)b=a(mb) \quad \text{for all}\ a \in A, \ b \in B \ \text{and}\ m
\in M.
$$

\item[{(\rm ii)}]
If $M$ is an $(A, B)$-bimodule and $N$ is a $(B, C)$-bimodule, then
$M \otimes _{B} N$ is naturally an $(A, C)$-bimodule, where the
bimodule structure is given by
$$
a(m\otimes n)c=(am)\otimes (nc) \quad \text{for all}\ a \in A, \ c
\in C, \ m \in M \ \text{and} \ n \in N.
$$

\item[{(\rm iii)}]
If $M$ is an $(A, B)$-bimodule and $N$ is an $(A, C)$-bimodule, then
the set of all left $A$-homomorphi\-sms from $M$ to $N$, denoted by
$\Hom_{lA} (M, N)$, is  naturally a $(B, C)$-bimodule, where
$$
(b\varphi c)(m)=(\varphi (mb))c \quad  \text{for all}\ b \in B, \ c
\in C, \ m \in M \ \text{and}\ \varphi \in \Hom_{lA} (M, N).
$$
Similarly if $M$ is an $(A, B)$-bimodule and $N$ is a $(C,
B)$-bimodule, then the set of all right $B$-homomorphisms from $M$
to $N$, denoted by $\Hom_{rB} (M, N)$, is naturally a $(C,
A)$-bimodule, where
$$
(c\psi a)(m)=c (\psi (am)) \quad  \text{for all}\ c \in C, \ a \in
A, \ m \in M   \ \text{and}\ \psi \in \Hom_{rB} (M, N).
$$

\item[{(\rm iv)}]
If $M$ is an $(A, B)$-bimodule, $N$ is a $(B, C)$-bimodule, and $L$
is an $(A, D)$-bimodule, then there exists a $(C, D)$-bimodule
isomorphism, the so-called {\em adjoint isomorphism},
$$
\Xi : \Hom_{lA} (M\otimes _{B} N, L) \stackrel{\cong}{\lra}
\Hom_{lB} (N, \Hom_{lA} (M, L))
$$
which is such that
$$
((\Xi(\phi))(n))(m) = \phi(m\otimes n)  \quad \text{for all}\ n \in
N, \ m \in M   \ \text{and}\ \phi \in \Hom_{lA} (M\otimes_B N,L).
$$

\item[{(\rm v)}] Similarly, if $M$ is an $(A, B)$-bimodule, $N$ is a $(B, C)$-bimodule, and $L$
is a $(D, C)$-bimodule, then there exists an `adjoint' $(D,
A)$-bimodule isomorphism
$$
\Theta : \Hom_{rC} (M\otimes _{B} N, L) \stackrel{\cong}{\lra}
\Hom_{rB} (M, \Hom_{rC} (N, L))
$$
which is such that
$$
((\Theta(\varphi))(m))(n) = \varphi(m\otimes n) \quad \text{for
all}\ m \in M, \ n \in N    \ \text{and}\ \varphi \in \Hom_{rC}
(M\otimes_B N,L).
$$

\item[{(\rm vi)}]
If $M$ is an $(A, B)$-bimodule, $N$ is a $(C, D)$-bimodule, and
$L$ is an $(A, D)$-bimodule,
then there exists a $(B, C)$-bimodule isomorphism
$$
\Omega : \Hom_{lA} (M, \Hom_{rD}(N, L)) \cong \Hom_{rD} (N,
\Hom_{lA} (M, L))
$$
for which
$$
((\Omega(\psi))(n))(m) = (\psi(m))(n) \quad \text{for all}\ m \in M,
\ n \in N    \ \text{and}\ \psi \in \Hom_{lA} (M, \Hom_{rD}(N, L)).
$$
\end{itemize}
\end{rmk}

Recall that $R$ denotes a commutative Noetherian ring of prime
characteristic $p$ and that $f: R \lra R$ denotes the Frobenius
homomorphism.  We shall only assume that $R$ is $F$-finite when this
is explicitly stated.

Let  $M$  be an $R$-module. We always regard $M$ as an $(R,
R)$-bimodule by $r\cdot m \cdot s = rsm$  for  $r, s \in R$  and  $m
\in M$. On the other hand, we define the $(R, R)$-bimodule $M_{f}$
to be $M_{f}=M$ as Abelian group with $(R, R)$-bimodule structure
defined by
$$
r\cdot m\cdot s=rs^pm \quad \text{for all~} r,s \in R \ \text{and} \  m \in M.
$$
Note that  the Frobenius map  $f : R \to R_f$  is a right $R$-module
homomorphism.

Similarly, we define the $(R, R)$-bimodule $_{f}M$ to be  $_{f}M=M$
as Abelian group with $(R, R)$-bimodule structure defined by
$$
r\cdot m\cdot s=r^p sm \quad \text{for all~} r,s \in R \ \text{and} \  m \in M.
$$

\begin{rmk}\label{rem}
Let  $M$  be an $R$-module.

\begin{itemize}
\item[{(\rm i)}]
By \ref{bimod}(ii), $M \otimes _R R_f$  has naturally a structure of  $(R, R)$-bimodule.
The action of $R$  is given by
$$
s\cdot (m \otimes r) \cdot s' =  srs'^p m \otimes 1 \quad \text{for
all~} r,s,s' \in R \ \text{and} \  m \in M.
$$
Thus there is an isomorphism
$M_{f}\cong M\otimes _{R} R_{f}$ as $(R, R)$-bimodules.

\item[{(\rm ii)}]
Similarly,  ${}_f R \otimes _R M$  is an $(R, R)$-bimodule with
action given by
$$
s\cdot (r \otimes m) \cdot s' = 1 \otimes s^p rs' m  \quad \text{for
all~} r,s,s' \in R \ \text{and} \  m \in M.
$$
There is an isomorphism  ${}_{f}M \cong {}_{f} R \otimes _{R} M$ of
$(R, R)$-bimodules.

\item[{(\rm iii)}]
By \ref{bimod}(iii), the Abelian group  $\Hom _{lR} (R_f , M)$
consisting of all left  $R$-homomorphisms from  $R_f$  to  $M$  is
an $(R, R)$-bimodule  with action of  $R$  given by
$$
(s\varphi s')(r) = (\varphi (r\cdot s))s' = rs^ps' \varphi (1) \quad
\text{for all~} r,s,s' \in R \ \text{and} \ \varphi \in \Hom
_{lR}(R_f, M).
$$
It is easy to see that
$\Hom_{lR}(R_{f}, M) \cong {}_{f}M$
as $(R, R)$-bimodules.

\item[{(\rm iv)}]
Similarly, the set $\Hom _{rR}(R_f, M)$  of all right
$R$-homomorphisms from  $R_f$  to  $M$ is an  $(R, R)$-bimodule with
$$
(s\psi s')(r) = s \psi (s'r) \quad \text{for all~} r,s,s' \in R \
\text{and} \ \psi \in \Hom _{rR}(R_f, M).
$$
\end{itemize}
\end{rmk}

We shall use a refinement of the following result.

\begin{lem} [Y. Yoshino {\cite[Lemma 3.6]{Yoshi94}}]\label{yoshino}
Suppose that $(R,\fm)$ is local and $F$-finite. Denote by $E$ the
injective envelope $E_R(R/\fm)$ of the simple $R$-module, which we
regard as a right $R$-module. Then there is a right $R$-module
isomorphism $E \stackrel{\cong}{\lra} \Hom_{rR}(R_f , E)$, where
$\Hom_{rR}(R_f, E)$ carries the right $R$-module structure described
in Remark\/ {\rm \ref{rem}(iv)}.
\end{lem}

We shall use the following refinement, in which it is not assumed
that $R$ is local.

\begin{lem}
\label{main lemma}
Suppose that $R$ is $F$-finite, and let $I$ be an injective $R$-module.
Then there is an $(R, R)$-bimodule isomorphism
$$
\Hom _{l R} (R_f, I) \stackrel{\cong}{\lra} \Hom_{rR}(R_f, I).
$$
\end{lem}

\begin{proof}
It is a consequence of the adjoint isomorphism of Remark
\ref{bimod}(v) that $\Hom_{rR}(R_f, I)$ is injective as right
$R$-module. On the other hand, by \ref{rem}(iii), we have an
isomorphism of $(R, R)$-bimodules $\Hom _{l R} (R_f, I) \cong {}_f
I$.

We can use the well-known decomposition theory for injective
$R$-modules due to E. Matlis (reviewed in, for example,
\cite[\S18]{HM}) to see that it is enough for us to prove the result
when $I = E_R(R/\fp)$ for a prime ideal $\fp$ of $R$, and so we
assume that this is so in the rest of the proof.

Since, for a prime ideal $\fq$ of $R$, each element of $E_R(R/\fq)$
is annihilated by some power of $\fq$, and multiplication by an
element $r \in R \setminus \fq$ provides an automorphism of
$E_R(R/\fq)$, it follows that $\Hom_{rR}(R_f, E_R(R/\fp))$ (with the
right $R$-module structure described in Remark\/ {\rm
\ref{rem}(iv)}) is isomorphic to a direct sum of $\mu$ copies of
$E_R(R/\fp)$. First we prove that the cardinal $\mu$ is exactly $1$.
Thus
$$
\Hom_{rR}(R_f, E_R(R/\fp)) \cong \bigoplus \mu
\left(E_{R}(R/\fp)\right),
$$
as right $R$-modules. We consider $(R_f)_{\fp}$ as the localization
of the right $R$-module $R_f$ at $\fp$ and write the resulting
action of $R_{\fp}$ on the right. (Thus $(r/s) \cdot (a/t) =
ra^p/st$ for $r \in R_{f}, \ a \in R$ and $s,t \in R \setminus
\fp$.) We can also endow this $(R_f)_{\fp}$ with a left
$R_{\fp}$-module structure under which
$$
\left(\frac{a}{t}\right)\cdot\left(\frac{r}{s}\right) =
\frac{t^{p-1}ar}{st} \quad \text{for all~}r \in R_{f}, \ a \in R \
\text{and} \ s,t \in R \setminus \fp.
$$
These two structures turn $(R_f)_{\fp}$ into an $(R_{\fp},
R_{\fp})$-bimodule, and then there is an $(R_{\fp},
R_{\fp})$-bimodule isomorphism $\beta : (R_f)_{\fp}
\stackrel{\cong}{\lra} (R_{\fp})_f$ for which $\beta(r/s) = r/s^p$
for all $r \in R_f$ and $s \in R \setminus \fp$. Since $R_f$ is
finitely generated as right $R$-module, there is a right
$R_{\fp}$-module isomorphism
$$
\left(\Hom_{rR}(R_f, E_R(R/\fp)) \right)_{\fp} \cong
\Hom_{rR_{\fp}}\!\left((R_f)_{\fp},(E_R(R/\fp))_{\fp}\right)
$$
when $\Hom_{rR}(R_f, E_R(R/\fp))$ (respectively
$\Hom_{rR_{\fp}}\!\left((R_f)_{\fp}, (E_R(R/\fp))_{\fp} \right)$) is
considered as a right $R$-module (respectively a right
$R_{\fp}$-module) via Remark \ref{bimod}(iii). One can use this
isomorphism, and the isomorphism $\beta$ above, to see that there is
a right $R_{\fp}$-module isomorphism
$$
\left(\Hom_{rR}(R_f, E_R(R/\fp)) \right)_{\fp} \cong
\Hom_{rR_{\fp}}\!\left((R_{\fp})_f, E_{R_{\fp}}(R_{\fp}/\fp
R_{\fp})\right).
$$
The last module is right $R_{\fp}$-isomorphic to
$E_{R_{\fp}}(R_{\fp}/\fp R_{\fp})$ by Lemma \ref{yoshino}. Therefore
$\mu = 1$.

We have thus shown that there is a right $R$-module isomorphism
$\varphi : I \to \Hom_{rR}(R_f, I)$. To finish the proof, we show
that this mapping  $\varphi$, regarded as a mapping ${}_f I \to
\Hom_{rR}(R_f, I)$, is actually a left $R$-module homomorphism, and
therefore an $(R, R)$-bimodule isomorphism.  For $z \in {}_fI$ and
$a \in R$, we have, for all $r \in R_f$,
\begin{align*}
\varphi (a\cdot z) (r) & = \varphi (za^p)(r) = (\varphi (z)a^p)(r)=
\varphi (z)(a^pr)\\ & = \varphi (z)(r \cdot a) = (\varphi (z)(r))a =
a(\varphi (z)(r)) = (a\varphi (z))(r),
\end{align*}
so that $\varphi (a\cdot z) = a \varphi (z)$. Therefore $\varphi$ is
a left $R$-homomorphism.
\end{proof}

\begin{rmk} If, in Lemma \ref{main lemma}, we drop the hypothesis
that $R$ is $F$-finite, then the conclusion is no longer always
true. For one example, let  $K$  be a countable field of
characteristic $p$ with $[K : K^p]$ infinite but countable, and set
$R =K$. We show now that $\Hom _{rK}(K_f, K) \not\cong {}_fK$ as
right $K$-modules. Assume that $\Hom _{rK}(K_f, K) \cong {}_fK$ as
right $K$-modules and seek a contradiction.

Let $\overline{K}$ be an algebraic closure of $K$, and let $K^{1/p}$
denote the subfield of $\overline{K}$ consisting of all $p$th roots
of elements of $K$. The assumption implies that $\Hom _{K}(K^{1/p},
K) \cong K^{1/p}$ as $K^{1/p}$-modules. In particular, $\Hom
_{K}(K^{1/p}, K)$ has countable dimension as a vector space over
$K$. Let $( \alpha _{n})_{n\in\N}$ be a $K$-basis of $K^{1/p}$, so
that $K^{1/p} = \bigoplus _{n\in\N} K \alpha _{n}$. Then $\Hom
_{K}(K^{1/p}, K) = \Hom_K\left(\bigoplus _{n\in\N} K \alpha _{n},
K\right) \cong \prod_{n\in\N} \Hom_K(K \alpha _{n},K)$, and this has
uncountable dimension as a vector space over $K$, and this is a
contradiction.
\end{rmk}

\begin{disc}
\label{ltf} The Frobenius skew polynomial ring   $R[x, f]$  was
defined in the Introduction. It follows from \cite[Lemma 1.3]{KS}
that extension of the $R$-module structure on an $R$-module $H$ to a
structure of left $R[x,f]$-module is equivalent to the provision of
an Abelian group homomorphism $\xi : H \lra H$ for which $\xi(rh) =
r^p\xi(h)$ for all $r \in R$ and $h \in H$. (In fact, $\xi$ and the
action of $x$ are related by the formula $\xi(h) = xh$ for all $h
\in H$.)

There is a bijective correspondence between
$\Hom_{lR}(R_f\otimes_RH,H)$ and
$$
\left\{ \xi \in \Endo_{\Z}(H) \ | \ \xi(rh) = r^p\xi(h)\text{~for
all~}r \in R\text{~and~}h \in H\right\}
$$
under which $\alpha\in \Hom_{lR}(R_f\otimes_RH,H)$ corresponds to $h
\mapsto \alpha(1 \otimes h)$. In view of this, we are going to use
the notation $(H, \alpha)$  to describe a left  $R[x, f]$-module
$\H$,  where  $H$  is the underlying  $R$-module and  $\alpha \in
\Hom_{lR}(R_{f} \otimes_{R} H,  H)$ is such that $xh =
\alpha(1\otimes h)$ for all $h \in H$.

Under the adjoint isomorphism of Remark \ref{bimod}(iv), an $\alpha
\in \Hom_{lR}(R_{f} \otimes_{R} H,  H)$ corresponds to an
$\widetilde{\alpha} \in \Hom _{lR} (H , \Hom _{lR}(R_f, H))$. Note
that $xh = (\widetilde{\alpha}(h))(1)$ for all $h \in H$. We write
$\H = (H, \alpha) = [H, \widetilde{\alpha}]$.

With such notation, a left  $R[x, f]$-homomorphism  $\H  =  (H,
\alpha) \to \H' = (H', \alpha ')$  of left $R[x,f]$-modules is an
$R$-homomorphism  $\varphi : H \to H'$ for which the diagram
$$
\begin{CD}
R_f \otimes _R H @>{\alpha}>> H \\
@V{1\otimes \varphi}VV         @V{\varphi}VV \\
R_f \otimes _R H' @>{\alpha '}>> H' \\
\end{CD}
$$
commutes.
\end{disc}

\begin{disc}
\label{rtf} Similarly, extension of the $R$-module structure on an
$R$-module $M$ to a structure of right $R[x,f]$-module is equivalent
to the provision of an Abelian group homomorphism $\xi : M \lra M$
for which $\xi(mr^p) = \xi(m)r$ for all $r \in R$ and $m \in M$. The
map $\xi$ and the action of $x$ are related by the formula $\xi(m) =
mx$ for all $m \in M$.

There is a bijective correspondence between
$\Hom_{rR}(M\otimes_RR_f,M)$ and
$$
\left\{ \xi \in \Endo_{\Z}(M)  \ | \  \xi(mr^p) = \xi(m)r\text{~for
all~}r \in R\text{~and~}m \in M\right\}
$$
under which $\beta\in \Hom_{rR}(M\otimes_RR_f,M)$ corresponds to $m
\mapsto \beta(m \otimes 1)$. In view of this, we are going to use
the notation $(M, \beta)$  to describe a right $R[x, f]$-module
$\M$, where  $M$  is the underlying  $R$-module and $\beta \in
\Hom_{rR}(M\otimes_RR_f,M)$ is such that $mx = \beta(m\otimes 1)$
for all $m \in M$.

Under the adjoint isomorphism of Remark \ref{bimod}(v), a $\beta \in
\Hom_{rR}(M\otimes_RR_f,M)$ corresponds to a $\widetilde{\beta} \in
\Hom _{rR} (M , \Hom _{rR}(R_f, M))$. Note that $mx =
(\widetilde{\beta}(m))(1)$ for all $m \in M$. We write $\M = (M,
\beta) = [M, \widetilde{\beta}]$.
\end{disc}

\begin{ntn}
\label{nt.1} We shall use ${}_{R[x,f]}\Mod$ to denote the category
of all left $R[x,f]$-modules and left $R[x,f]$-homomorphisms between
them, and $\Mod_{R[x,f]}$ to denote the category of all right
$R[x,f]$-modules and right $R[x,f]$-homomorphisms between them.
\end{ntn}

\begin{exs}
(i) The Frobenius endomorphism  $f : R \to R$  induces a left
$R$-module homomorphism  $\alpha : R _f \otimes_R R \to R$  for
which $\alpha (a \otimes b) = a f(b) = ab^p$ for all $a \in R_f$ and
$b \in R$. This therefore yields the left $R[x, f]$-module  $(R,
\alpha)$, in which we have $x r = r^p$ for all $r \in R$.

Let  $c \in R$  be any element. Then there is a left $R$-module
homomorphism $\alpha_c : R_f \otimes R \to R$  such that  $\alpha
_c(a \otimes b) = cab^p$  for  $a  \in R_f$ and $b \in R$. Thus we
obtain a left $R[x,f]$-module  $(R, \alpha _c)$, in which $x  r =
cr^p$ for all $r \in R$. It is straightforward to check that $(R,
\alpha) \cong (R, \alpha _c)$ as left $R[x,f]$-modules if and only
if  $c$ is a unit in $R$ possessing a $(p-1)$th root in $R$. Thus it
is possible for there to be many left $R[x,f]$-modules with the same
underlying $R$-module.

(ii) Suppose that our ring $R$ is reduced and that we are given a
non-trivial $R^p$-homomorphism $\pi : R \to R^p$. (In the case where
$R$ is $F$-finite and $F$-pure, we can find such a $\pi$ that is a
surjective mapping, because $R^p$ is a direct summand of $R$ as an
$R^p$-module: see \cite[Corollary 5.3]{HocRob76}.) In this
situation, we have a right $R$-module homomorphism $\beta : R
\otimes_R R_f \to R$ for which $\beta (a \otimes b) = \pi(ab)^{1/p}$
for all $a\in R$ and $b \in R_f$. This yields a right $R[x,
f]$-module  $(R, \beta)$, in which we have $r x = \pi (r)^{1/p}$ for
all $r \in R$.
\end{exs}

We have shown in \ref{main lemma} that whenever $R$ is $F$-finite
and $I$ is an injective $R$-module, there is an $(R,R)$-bimodule
isomorphism $\Psi:  {}_f I \to \Hom _{rR} (R_f, I)$; of course,
$\Psi$ is, in particular, a right  $R$-module homomorphism.
Therefore we have the following as a corollary to \ref{main lemma}.

\begin{cor}
\label{mainlemcor} Suppose that $R$ is $F$-finite, and let $I$ be an
injective $R$-module. Then there is a right $R[x,f]$-module  ${\bf
I} = [I, \Psi]$ which has $I$  as underlying $R$-module, and is such
that $\Psi : {}_f I \to \Hom _{rR} (R_f, I)$  is an $(R,R)$-bimodule
isomorphism. Note that $zx = (\Psi(z))(1)$ for all $z \in I$.
\end{cor}

\begin{lem}
\label{ann} Let the situation and notation be as in Corollary\/ {\rm
\ref{mainlemcor}}, and consider the right $R[x,f]$-module  ${\bf I}
= [I, \Psi]$. Then ${\bf I}$ has the following property: if $z \in
{\bf I}$ is such that, for a fixed $n \in \nn$, we have $zrx^n = 0$
for all $r \in R$, then $z = 0$.
\end{lem}

\begin{proof} The claim is clear when $n = 0$, and we deal now with
the case where $n=1$. We have $$0 = (zr)x = \left(\Psi(zr)\right)(1)
= \left((\Psi(z))r\right)(1) = (\Psi(z))(r) \quad \text{for all $r
\in R$}.$$ Therefore $\Psi(z) = 0$, so that $z = 0$ because $\Psi$
is an isomorphism.

Now suppose, inductively, that $n \in \N$ with $n >1$, and that the
claim has been proved for all smaller values of $n$. Suppose that
$zrx^n = 0$ for all $r \in R$. Then $(zrx^{n-1})sx =
zrs^{p^{n-1}}x^n = 0$ for all $r,s \in R$. It follows from the case
where $n = 1$ that $zrx^{n-1} = 0$ for all $r \in R$; it then
follows from the inductive hypothesis that $z = 0$.
\end{proof}

\begin{disc}\label{constrfunct}
Throughout the rest of this section, assume that our ring $R$ is
$F$-finite, and let $I$ be an injective $R$-module. We fix a right
$R[x, f]$-module structure on  $I$ as in Corollary {\rm
\ref{mainlemcor}}, so that ${\bf I} = [I, \Psi]$ is a right
$R[x,f]$-module with $\Psi : {}_f I \to \Hom _{rR} (R_f, I)$ an
$(R,R)$-bimodule isomorphism. We denote by $(-)^{\vee}$ the duality
functor determined by  $I$, so that $X ^{\vee} = \Hom _R (X, I)$ for
each $R$-module  $X$.

Now suppose we are given a left $R[x,f]$-module  $\H =  (H, \alpha)$
with $\alpha \in \Hom _{lR}(R_f \otimes _R H, H)$.

\begin{enumerate}\item
Here we produce a right $R[x,f]$-module structure on $H^{\vee}$.

First apply the functor $(-)^{\vee}$ to the left $R$-homomorphism
$\alpha :R_{f}\otimes _{R} H \to H$: the result is a right
$R$-homomorphism $ \alpha ^{\vee}  :   H^{\vee}  \to \Hom_{lR}(R_{f}
\otimes_{R} H, I). $ But there is an $(R,R)$-bimodule isomorphism
$\Hom_{lR}(R_{f} \otimes_{R} H, I) \stackrel{\cong}{\lra}
\Hom_{lR}(H, \Hom_{lR}(R_{f}, I))$ given by Remark \ref{bimod}(iv),
and use of the $(R,R)$-bimodule isomorphism $\Psi$ produces a
further $(R,R)$-bimodule isomorphism $$\Hom_{lR}(H, \Hom_{lR}(R_{f},
I)) \stackrel{\cong}{\lra} \Hom_{lR}(H, \Hom_{rR}(R_{f}, I)).$$ In
addition, Remark \ref{bimod}(vi) provides an $(R,R)$-bimodule
isomorphism
$$\Hom_{lR}(H, \Hom_{rR}(R_{f}, I)) \stackrel{\cong}{\lra} \Hom_{rR}(R_{f},
\Hom_{lR}(H, I)).$$ Composition of these therefore yields a right
$R$-homomorphism $\gamma: H^{\vee} \lra \Hom_{rR}(R_{f}, H^{\vee})$,
and we shall denote by $D(\alpha)$ the right $R$-homomorphism
$H^{\vee} \otimes _R R_f \to  H^{\vee}$ that corresponds to $\gamma$
under the adjoint isomorphism of Remark \ref{bimod}(v). (Note that
$H^{\vee} = \Hom _{lR} (H, I) = \Hom _{rR} (H, I)$.)

Thus $D(\alpha)$ makes $H^{\vee}$ into a right $R[x,f]$-module. We
define $$\D (\H) = \D (H, \alpha) := (H ^{\vee}, D(\alpha)) = [H
^{\vee},\gamma].$$ It is straightforward to use the above definition
of $\gamma$ to check that
\begin{equation}\label{D}
(D(\alpha)(m \otimes r))(h) = \left(\Psi(m(\alpha(1\otimes
h)))\right)(r) \quad\mbox{for all~}m \in H^{\vee},\ r \in R_f
\mbox{~and~} h \in H.
\end{equation}

\item Now let
$\H' = (H', \alpha ')$ be a second left $R[x,f]$-module and let
$\varphi: \H \to \H'$ be a left  $R[x,f]$-homomorphism. Thus
$\varphi$ is an $R$-homomorphism $H \to H'$  which makes the diagram
$$
\begin{CD}
R_{f}\otimes_{R} H @>{\alpha}>> H\\
@V{\Id\otimes \varphi}VV @VV{\varphi}V \\
R_{f}\otimes_{R} H' @>{\alpha '}>> H'
\end{CD}
$$
commute. It is straightforward to check that the diagram
$$
\begin{CD}
H^{\vee} @>{\alpha^{\vee}}>> \Hom_{lR}( R_{f}\otimes_{R} H, I)
@>\cong>>  \Hom_{rR}(R_{f}, \Hom_{l R}(H, I))\\
@A{\varphi^{\vee}}AA @AA{(\Id\otimes \varphi)^{\vee}}A @A{\Hom (R_f, \varphi^{\vee})}AA\\
H'^{\vee}@>{\alpha '^{\vee}}>> \Hom_{lR}(R_{f}\otimes_{R} H', I)
@>\cong>>  \Hom_{rR}(R_{f}, \Hom_{l R}(H',I))~,
\end{CD}
$$
in which the upper horizontal isomorphism is the one used in the
construction in part (i) and the lower horizontal isomorphism is the
corresponding one for $\H'$, commutes. Therefore $\varphi^{\vee} :
H'^{\vee} \to H^{\vee}$ defines a right $R[x,f]$-homomorphism $\D(\H
') \to \D (\H)$, which we denote by $\D (\varphi)$.
\end{enumerate}
\end{disc}

\begin{prop}\label{lt-rt} Let the situation and notation be as in
Discussion\/ {\rm \ref{constrfunct}}. There is a contravariant
functor $\D : {}_{R[x,f]}{\mathrm{Mod}} \to {\mathrm{Mod}}_{R[x,f]}$
which maps a left  $R[x,f]$-module  $(H, \alpha)$  to $(H^{\vee},
D(\alpha))$ where  $D(\alpha)$  is given by  {\rm (\ref{D})} in
Discussion\/ {\rm \ref{constrfunct}(i)}.
\end{prop}

\begin{prop}\label{action} Let the situation and notation be as in
Discussion\/ {\rm \ref{constrfunct}} and let $\H = (H, \alpha)$  be
a left $R[x,f]$-module. The right $R[x,f]$-module structures on
${\bf I}$ and $\D(\H) = (\Hom_{lR}(H,I),D(\alpha))$ are such that
\begin{equation}\label{right-act}
(m  x)(h) = (m (x h)) x \quad\text{for all $m \in H^{\vee} =
\Hom_{lR}(H,I)$ and $h \in H$}.
\end{equation}
\end{prop}

\begin{proof} Recall that the right $R[x,f]$-module structure on $I$
is given by $\Psi : {}_f I \to \Hom _{rR}(R_f, I)$, so that $zx =
(\Psi(z))(1)$ for all $z \in I$.

Let $m \in H^{\vee} = \Hom_{lR}(H,I)$ and $h \in H$. By (1) in
Discussion \ref{constrfunct}(i), we have
$$
(mx)(h) = (D(\alpha)(m\otimes 1))(h) = \Psi(m(\alpha(1\otimes
h)))(1) = (m(\alpha(1\otimes h)))x = (m (x h)) x.
$$
\end{proof}

We now provide the right $R[x,f]$-module analogue of Discussion
\ref{constrfunct}.

\begin{disc}\label{otherfunct} The hypotheses and notation are as in
Discussion \ref{constrfunct}. Let  $\M = (M, \beta)$  be a right
$R[x, f]$-module, where $\beta \in \Hom _{rR} (M \otimes _R R_f,
M)$.

\begin{enumerate}\item
Here we produce a left $R[x,f]$-module structure on $M^{\vee}$.

First apply the functor $(-)^{\vee}$ to the right $R$-homomorphism
$\beta :M \otimes _R R_f \to M$: the result is a left
$R$-homomorphism $ \beta ^{\vee}  :   M^{\vee}  \to \Hom_{rR}(M
\otimes _R R_f, I). $ But there is an $(R,R)$-bimodule isomorphism
$\Hom_{rR}(M \otimes _R R_f, I)) \stackrel{\cong}{\lra} \Hom_{rR}(M,
\Hom_{rR}(R_{f}, I))$ given by Remark \ref{bimod}(v), and use of the
$(R,R)$-bimodule isomorphism $\Psi^{-1}$ produces a further
$(R,R)$-bimodule isomorphism
$$\Hom_{rR}(M, \Hom_{rR}(R_{f},
I)) \stackrel{\cong}{\lra} \Hom_{rR}(M, \Hom_{lR}(R_{f}, I)).$$ In
addition, Remark \ref{bimod}(vi) provides an $(R,R)$-bimodule
isomorphism
$$\Hom_{rR}(M, \Hom_{lR}(R_{f}, I))  \stackrel{\cong}{\lra}
\Hom_{lR}(R_{f}, \Hom_{rR}(M, I)).$$ Composition of these therefore
yields a left $R$-homomorphism $\delta: M^{\vee} \lra
\Hom_{lR}(R_{f}, M^{\vee})$, and we shall denote by $D'(\beta)$ the
left $R$-homomorphism $R_f \otimes _RM^{\vee} \to M^{\vee}$ that
corresponds to $\delta$ under the adjoint isomorphism of Remark
\ref{bimod}(iv). (Note that $M^{\vee} = \Hom _{lR} (M, I) = \Hom
_{rR} (M, I)$.) Thus $D'(\beta)$ makes $M^{\vee}$ into a left
$R[x,f]$-module. We define $$\D' (\M) = \D (M, \beta) := (M ^{\vee},
D'(\beta)) = [M ^{\vee},\delta].$$ It is straightforward to use the
above definition of $\delta$ to check that
\begin{equation}\label{D'}
(D'(\beta)(r \otimes h))(m) = \left(\Psi^{-1}(r' \mapsto h(\beta(m
\otimes r')))\right)r \quad\mbox{for all~}h \in M^{\vee},\ r \in R_f
\mbox{~and~} m \in M.
\end{equation}

\item Now let
$\M' = (M', \beta ')$ be a second right $R[x,f]$-module and let
$\psi: \M \to \M'$ be a right  $R[x,f]$-homomorphism. An argument
similar to that in Discussion \ref{constrfunct}(ii) shows that
$\psi^{\vee} : M'^{\vee} \to M^{\vee}$ defines a left
$R[x,f]$-homomorphism $\D'(\M ') \to \D '(\M)$, which we denote by
$\D '(\psi)$.
\end{enumerate}
\end{disc}

\begin{prop}\label{rt-lt} Let the situation and notation be as in
Discussion\/ {\rm \ref{otherfunct}}. There is a contravariant
functor $\D ': {\mathrm{Mod}}_{R[x,f]} \to
{}_{R[x,f]}{\mathrm{Mod}}$ which maps a right  $R[x,f]$-module  $(M,
\beta)$  to $(M^{\vee}, D' (\beta))$ where  $D' (\beta)$  is given
by {\rm (\ref{D'})} in Discussion\/ {\rm \ref{otherfunct}(i)}.
\end{prop}

\begin{prop}\label{left-action}
Let the situation and notation be as in Discussion\/ {\rm
\ref{otherfunct}}. Let $\M = (M, \beta)$  be a right
$R[x,f]$-module, so that $\D' (\M) = (M^{\vee}, D' (\beta))$  is a
left $R[x,f]$-module by Proposition\/ {\rm \ref{rt-lt}}. The left
action of $x$ on $M^{\vee}$ can be described as follows: for $h \in
M^{\vee}$, the result $xh$ of multiplying $h$ on the left by $x$ is
the unique $h' \in M^{\vee}$ for which
\begin{equation}\label{left-act}
(h'(m))rx = h (mrx) \quad\text{for all $m \in M$ and $r\in R$}.
\end{equation}
\end{prop}

\begin{proof} First of all,
\begin{align*}
((xh)(m))rx & = \left((D'(\beta)(1 \otimes h))(m)\right)rx =
\left(\Psi^{-1}(r' \mapsto h(\beta(m \otimes r')))\right)rx\\ & =
\left(\Psi\left(\left(\Psi^{-1}(r' \mapsto h(\beta(m \otimes
r')))\right)r\right)\right)(1) = \left((r' \mapsto h(\beta(m \otimes
r')))r\right)(1) = h(\beta(m \otimes r))\\ & = h(\beta(m r\otimes
1)) = h(mrx).
\end{align*}
It therefore remains for us to show that if $h' \in M^{\vee}$ is
such that $(h'(m))rx = h (mrx)$ for all $m \in M$ and $r\in R$, then
$h' = xh$. It is therefore enough for us to show that if $h'' \in
M^{\vee}$ is such that $(h''(m))rx = 0$ for all $m \in M$ and $r\in
R$, then $h'' = 0$. However, this is easy, because Lemma \ref{ann}
shows that $h''(m) = 0$ for all $m \in M$.
\end{proof}

Propositions \ref{lt-rt} and \ref{rt-lt} prepare the ground for
several subsequent results in this paper.

\begin{prop}\label{natural transformation}
Let the situation and notation be as in Propositions\/ {\rm
\ref{lt-rt}} and\/ {\rm \ref{rt-lt}}, so that $R$  is $F$-finite and
$I$ is an injective $R$-module  with fixed $(R, R)$-bimodule
isomorphism $\Psi : {}_f I \to \Hom _{rR}(R_f, I)$.

For each $R$-module $G$, we write $G^{\vee} = \Hom _R (G, I)$ as
before. Let $\omega_G : G \lra (G^{\vee})^{\vee}$ be the natural
`evaluation' $R$-homomorphism for which $\omega_G(g)(h) = h(g)$ for
all $h \in G^{\vee}$ and $g \in G$. Recall that, as $G$ varies
through the category ${}_R\Mod$ of all $R$-modules and
$R$-homomorphisms, the $\omega_G$ constitute a natural
transformation from the identity functor on ${}_R \Mod$ to the
functor $((-)^\vee)^\vee$.

\begin{enumerate}
\item
If  $\H =(H, \alpha)$ is a left $R[x,f]$-module, then $\omega_H$ is
a left $R[x,f]$-module homomorphism from  $\H$  to  $\D'(\D(\H))$.
As $\H$ varies through ${}_{R[x,f]}\Mod$, the $\omega_H$ constitute
a natural transformation from the identity functor on that category
to the functor $\D' \circ \D$.

\item
If  $\M = (M, \beta)$ is a right $R[x,f]$-module, then $\omega_M$ is
a right $R[x,f]$-module homomorphism from $\M$ to $\D(\D'(\M))$. As
$\M$ varies through $\Mod_{R[x,f]}$, the $\omega_M$ constitute a
natural transformation from the identity functor on that category to
the functor $\D \circ \D'$.
\end{enumerate}
\end{prop}

\begin{proof}
(i) In view of Propositions \ref{lt-rt} and \ref{rt-lt}, it only
remains for us to show that, for a left $R[x,f]$-module $\H = (H,
\alpha)$, the $R$-homomorphism $\omega_H : H \to (H^{\vee})^{\vee}$
is actually a left $R[x,f]$-module homomorphism. To this end, we
compare, for an $h \in H$, the elements $\omega_H (xh)$ and $x
(\omega_H (h))$. Now, $x \omega_H (h)$ is, by (\ref{left-act}) in
Proposition \ref{left-action}, the unique element $h' \in
(H^{\vee})^{\vee}$ that satisfies
$$
\left(h'(m)\right)rx  = \omega _H (h)(mr  x) \quad \text{for all $m
\in H ^{\vee}$  and  $r \in R$}.
$$
It is enough to show that  $h' = \omega _H(x h)$  satisfies this.
But
$$
(\omega _H (x h)(m)) rx = (m(x  h)) rx = (r(m (x h ))) x = (m(rx
h))x = ((mr)(xh))x = (mr x)(h),
$$
where we have used (\ref{right-act}) in \ref{action}  for the last
equality. Since $(mr x)(h) = \omega _H (h) (mrx)$, the proof of part
(i) is complete.

(ii) In view of Propositions \ref{lt-rt} and \ref{rt-lt}, it only
remains for us to show that, for a right $R[x,f]$-module $\M = (M,
\beta)$, the $R$-homomorphism $\omega_M : M \to (M^{\vee})^{\vee}$
is actually an  $R[x,f]$-homomorphism. To this end, we compare, for
an $m \in M$, the elements $\omega_M (m x)$ and $(\omega_M (m)) x$.

Now, for all $h \in M ^{\vee}$, we have
\begin{align*}
\left((\omega_M(m)) x \right)(h) &= \left(\omega _M (m)
(x h)\right) x  &\text{(by (\ref{right-act}) in \ref{action})} \\
& = \left((x h)(m))\right) x & \\
&= h(mx) &\text{(by (\ref{left-act}) in \ref{left-action})} \\
&= \left(\omega_M(m x)\right)(h). &
\end{align*}
Hence $\omega_M(m  x) = (\omega_M(m)) x$.
\end{proof}

\begin{rmk}
\label{rl.6} Let  $R'$  be a general commutative Noetherian ring and
let $I$ be an injective $R'$-module. For each $R'$-module  $M$ we
write $M ^{\vee} := \Hom_{R'} (M, I)$ and  denote by $\omega _M$ the
natural evaluation mapping  $M \to (M ^{\vee})^{\vee}$ defined by
$\omega_M (m)(h) = h(m)$ for all $h \in M^{\vee}$ and $m \in M$. We
say that $M$  is {\it $I$-reflexive} if $\omega _M$  is an
isomorphism. It is routine to check that, for an $R'$-module $M$,
the composition
$$
(\omega_M)^{\vee} \circ \omega_{M^{\vee}} : M^{\vee} \lra M^{\vee}
$$
is the identity map. Therefore, if $M$ is $I$-reflexive, then so too
is $M^{\vee}$. It is easily verified that the full subcategory of
${}_{R'}\Mod$ consisting of all $I$-reflexive modules is closed
under finite direct sums, direct summands and extensions. But in
general it is not a Serre subcategory of  ${}_{R'}\Mod$, as can be
seen by consideration of the case where $R'$ is a Noetherian
integral domain that is not a field and $I$ is taken to be the
quotient field of $R'$.

Suppose, in addition, that $(R',\fm)$ is (Noetherian) local and
complete. Choose $I = E := E_{R'}(R'/\fm)$, so that $( - )^{\vee}$
becomes the {\em Matlis-duality functor\/} $\Hom_{R'}( - , E)$. In
this case, $E$-reflexive modules are called {\em
Matlis-reflexive\/}. It is well known that all Noetherian
$R'$-modules and all Artinian $R'$-modules are Matlis-reflexive, and
that  $(- )^{\vee}$  provides a duality between the category of all
Noetherian $R'$-modules (and all $R'$-homomorphisms between them)
and the category of all Artinian $R'$-modules (and all
$R'$-homomorphisms between them). Furthermore it was proved by E.
Enochs \cite[Proposition 1.3]{Enochs} that an $R'$-module $M$ is
Matlis-reflexive if and only if  it can be embedded into a short
exact sequence
$$
\begin{CD}
0 @>>> N @>>> M @>>> A @>>> 0,
\end{CD}
$$
in which  $A$  is an Artinian $R'$-module and  $N$  is a Noetherian
$R'$-module. Therefore, the full subcategory of  ${}_{R'}\Mod$
consisting of all Matlis-reflexive modules is an Abelian category
itself, and is actually the smallest Serre subcategory of
${}_{R'}\Mod$ that contains all Noetherian modules and all Artinian
modules.
\end{rmk}

Let the situation and notation be as in Proposition {\rm
\ref{natural transformation}}, so that $R$  is $F$-finite and  $I$
is an injective $R$-module with a fixed $(R, R)$-bimodule
isomorphism $\Psi : {}_f I \to \Hom _{rR}(R_f , I)$. Let
$\mathcal{L}_I$ be the category of all left $R[x,f]$-modules which
are $I$-reflexive as $R$-modules, and all left
$R[x,f]$-homomorphisms between them. Similarly let $\mathcal{R}_I$
be the category of right $R[x,f]$-modules which are $I$-reflexive as
$R$-modules, and all right $R[x,f]$-homomorphisms between them.

In general, $\mathcal{L}_I \subseteq {}_{R[x,f]}\Mod$  and
 $\mathcal{R}_I \subseteq \Mod_{R[x,f]}$  are full subcategories,
 which are closed under finite direct sums, direct summands and extensions.
If, in addition, $(R, \fm)$  is local and complete and $I =
E_R(R/\fm)$, then
 $\mathcal{L}_I$  and  $\mathcal{R}_I$  are Abelian categories by Remark \ref{rl.6}.

From the definitions of the functors $\D$  and  $\D'$, it is easy to
use Remark \ref{rl.6} to see that they induce functors  $\D :
\mathcal{L}_I \to (\mathcal{R}_I)^{\text{op}}$  and
 $\D' : \mathcal{R}_I \to (\mathcal{L}_I)^{\text{op}}$.

The following theorem is the main result of this paper.


\begin{thm}\label{main thm}
Let the situation and notation be as in {\rm \ref{natural
transformation}}, so that $R$  is $F$-finite and  $I$  is an
injective $R$-module with a fixed $(R, R)$-bimodule isomorphism
$\Psi : {}_f I \to \Hom _{rR}(R_f , I)$. Then the functors  $\D :
\mathcal{L}_I \to (\mathcal{R}_I)^{\text{{\rm op}}}$  and
 $\D' : \mathcal{R}_I \to (\mathcal{L}_I)^{\text{{\rm op}}}$  are inverse
 equivalences of categories.
\end{thm}

\begin{proof}
For any  $\H =(H, \alpha) \in \mathcal{L}_I$, the evaluation mapping
$\omega _H : H \to (H^{\vee})^{\vee}$  is an $R[x,f]$-isomorphism,
by Proposition \ref{natural transformation}(i). Therefore the
natural transformation  $\omega : \Id \to \D' \circ \D$ of
\ref{natural transformation}(i) is a natural equivalence of functors
on $\mathcal{L}_I$. Similarly, $\omega: \Id \to \D\circ \D'$  is a
natural equivalence of functors on $\mathcal{R}_I$.
\end{proof}

From this theorem we have the following corollary in the complete
local case.

\begin{cor}
Assume that  $(R, \fm)$  is $F$-finite, complete and local, and let
$I = E := E_R(R/\fm)$. We fix an $(R,R)$-bimodule isomorphism  $\Psi
: {}_fE \to \Hom _{rR}(R_f, E)$.
\begin{enumerate}
\item
It follows from Theorem\/ {\rm \ref{main thm}} that $\D$ and $\D'$
are inverse equivalences between the category of left
$R[x,f]$-modules that are Artinian as $R$-modules and the category
of right $R[x,f]$-modules that are Noetherian as $R$-modules.

\item
Similarly, $\D$ and $\D'$ are inverse equivalences between the
category of right $R[x,f]$-modules that are Artinian as $R$-modules
and the category of left $R[x,f]$-modules that are Noetherian as
$R$-modules.
\end{enumerate}
\end{cor}

\section{\sc Graded annihilators}
\label{ga}

Let $\fB$ be a subset of $R[x,f]$. It is easy to see that $\fB$ is a
graded two-sided ideal of $R[x,f]$ if and only if there is an
ascending chain $(\fb_n)_{n \in \nn}$ of ideals of $R$ (which must,
of course, be eventually stationary) such that $\fB =
\bigoplus_{n\in\nn}\fb_n x^n$. In particular, note that $R[x,f]x^t =
\bigoplus_{i \geq t} Rx^i$ is a graded two-sided ideal of $R[x,f]$,
for each $t \in \nn$.

\begin{defs}
\label{ga.1} Let $\H = (H, \alpha)$ denote a left $R[x,f]$-module
and let $\M = (M, \beta)$  denote a right $R[x,f]$-module; let $\fB$
be a two-sided ideal of $R[x,f]$.

The {\em annihilator of\/ $\M$\/} will be denoted by $\ann \M
_{R[x,f]}$. Thus
$$
\ann \M _{R[x,f]} = \{ \theta \in R[x,f] \ | \  g \theta = 0
\mbox{~for all~} g \in M \},
$$
and this is a two-sided ideal of $R[x,f]$.
The annihilator $\ann_{R[x,f]} \H$ of $\H$ is defined similarly;
it is also a two-sided ideal of $R[x,f]$.

We define the {\em graded annihilator\/} $\grann \M _{R[x,f]}$ of
the right $R[x,f]$-module $\M$ by
$$
\grann \M _{R[x,f]} = \left\{ \sum_{i=0}^n r_ix^i \in R[x,f] \ \Big|
\ n \in \nn \mbox{~and~} r_i \in R,\, r_ix^i \in \ann \M _{R[x,f]}
\mbox{~for all~} i = 0, \ldots, n\right\}.
$$
Thus $\grann \M _{R[x,f]}$ is the largest graded two-sided ideal of
$R[x,f]$ contained in $\ann \M _{R[x,f]}$.

The graded annihilator of $\H$ is defined similarly: it is the
largest graded two-sided ideal of $R[x,f]$ that annihilates $\H$.
See \cite[1.5]{ga}. Recall also that $\ann_{\H} \fB$ denotes the
$R[x,f]$-submodule of $\H$ given by $$\ann_{\H} \fB = \{h \in \H \ |
\ \theta  h = 0 \mbox{~for all~} \theta \in \fB\}.$$

Observe that $M \fB$ is an $R[x,f]$-submodule of $\M$. For $t \in
\nn$, we have $M  R[x,f]x^t = \{m  x^t : m \in M\}$, and we shall
therefore denote this $R[x,f]$-submodule of $\M$ by $\M x^t$. We
shall say that $\M$ is {\em $x$-divisible\/} precisely when $\M = \M
x$.

Recall (from \cite[1.2]{ga}) that $\H = (H, \alpha)$ is said to be
{\em $x$-torsion-free\/} if $x h = 0$, for $h \in H$, only when $h =
0$. The set $\Gamma_x(H) := \left\{ h \in H \ | \ x^j h = 0
\mbox{~for some~} j \in \N \right\}$ is an $R[x,f]$-submodule of
$\H$, called the  \/{\em $x$-torsion submodule} of $\H$.
\end{defs}

We are now going to compare, in the situation of Theorem \ref{main
thm}, the graded annihilators of a left $R[x,f]$-module $\H$ and the
right $R[x,f]$-module $\D(\H)$, and also the graded annihilators of
a right $R[x,f]$-module $\M$ and the left $R[x,f]$-module $\D'
(\M)$.

Recall that an {\em injective cogenerator for $R$\/} is an injective
$R$-module $I$ such that $\Hom_R(G, I) \neq 0$ for every non-zero
$R$-module $G$. See \cite[p.\ 46]{SV}. It should be remarked that if
$I$  is an injective cogenerator, then the evaluation map  $\omega
_G : G \lra  (G^{\vee})^{\vee}$  is a monomorphism for all
$R$-modules $G$.

\begin{prop}
\label{ga.2} Let the situation and notation be as in Theorem {\rm
\ref{main thm}}. Let $\H = (H, \alpha)$ be a left $R[x,f]$-module
and $\M = (M, \beta)$ be a right $R[x,f]$-module. Then
\begin{enumerate}
\item
$\grann_{R[x,f]} \H \subseteq \grann \D (\H)_{R[x,f]}$;
\item
$\grann \M _{R[x,f]} \subseteq \grann_{R[x,f]}\D'(\M)$;
\item
if $I$ is an injective cogenerator for $R$, we have
$$
\grann_{R[x,f]} \H  = \grann \D (\H)_{R[x,f]} \quad \mbox{and} \quad
\grann \M _{R[x,f]} = \grann_{R[x,f]}\D '(\M) \mbox{;}
$$
\item
in particular, in the special case in which $(R,\fm)$ is local, and
$I$ is taken to be $E_R(R/\fm)$, we have
$$
\grann_{R[x,f]} \H = \grann \D (\H)_{R[x,f]}
\quad \mbox{and} \quad \grann \M _{R[x,f]} = \grann_{R[x,f]}\D'(\M).
$$
\end{enumerate}
\end{prop}

\begin{proof}
(i) Recall from (\ref{right-act}) in \ref{action} that the right
action of $R[x,f]$ on $\D (\H) =(H ^{\vee}, D(\alpha))$ is such that
$(m x)(h) = \left(m(xh)\right)x$ for all $m \in H ^{\vee}$ and all
$h \in H$; an easy inductive argument shows that $(m x^n)(h) =
\left(m(x^n h)\right) x^n$ for all $n \in \N$.

Now let $r \in R$ and $n \in \nn$ be such that $rx^n  H = 0$. We
show that $rx^n$ annihilates the right $R[x,f]$-module $\D
(\H)=(H^{\vee}, D(\alpha))$. Let $m \in H^{\vee}$ and $h \in H$.
Then, by the preceding paragraph,
$$
(m rx^n)(h) = ((mr)x^n)(h) = \left((mr)(x^n h)\right) x^n =
\left(m(rx^n h)\right)x^n = 0.
$$
It follows that $\grann_{R[x,f]} \H \subseteq \grann \D(\H)_{R[x,f]}$.

(ii) Let $r \in R$ and $n \in \nn$ be such that $M rx^n = 0$. We
show that $rx^n$ annihilates the left $R[x,f]$-module $\D' (\M) =
(M^{\vee}, D(\beta))$. This is clear when $n = 0$, and so we suppose
that $n> 0$. Let $h \in M^{\vee}$ and $m \in M$. Then recall from
(\ref{left-act}) in \ref{left-action} that $((xh)(m)) r'x = h(m
r'x)$ for all $m \in M$ and $r' \in R$. An easy inductive argument
shows that
$$
((x^n  h)(m))  r'x^n = h (m r'x ^n) \quad \text{for all $m \in M$
and $r' \in R$.}
$$
It follows from this that, for $r \in R$,
$$
((rx^n  h)(m))  r'x^n = h (m rr'x ^n) \quad \text{for all $m \in M$
and $r' \in R$.}
$$
But, since  $m rr'x^n = mr' rx^n =0$, we have  $rx^n h(m) = 0$ for
all $m \in M$, by Lemma \ref{ann}. Therefore $rx^nh = 0$ and $rx^n$
annihilates the left $R[x,f]$-module $\D' (\M)$. Hence
$$
\grann \M _{R[x,f]} \subseteq \grann_{R[x,f]} \D'(\M).
$$

(iii) By parts (i) and (ii), we have $\grann _{R[x,f]} \H \subseteq
\grann \D(\H)_{R[x,f]} \subseteq \grann _{R[x,f]} \D'\circ \D (\H)$.
However, since $I$ is an injective cogenerator for $R$, the
homomorphism of left $R[x,f]$-modules $\omega_H : H \lra
(H^{\vee})^{\vee}$ is actually an $R[x,f]$-monomorphism, and so it
follows that
$$
\grann _{R[x,f]}\D' \circ \D (\H) \subseteq \grann_{R[x,f]} \H.
$$
The first equality is therefore proved.
The second is proved similarly.

(iv)
This is a special case of part (iii), because, when $(R,\fm)$ is local, $E_R(R/\fm)$ is an injective cogenerator for $R$.
\end{proof}

\begin{prop}
\label{ga.3}
Let the situation and notation be as in\/ {\rm \ref{main thm}}, and assume in addition that $(R ,\fm)$ is local and complete and that $I$ is taken to be $E := E_R(R/\fm)$.
Let $\fB$ be a graded two-sided ideal of $R[x,f]$.

\begin{enumerate}
\item
Let $\H = (H, \alpha)$ be a left $R[x,f]$-module that is Matlis-reflexive as  $R$-module.
Let $\iota : \ann_{\H} \fB \lra \H$ denote the inclusion $R[x,f]$-monomorphism.  Then the induced homomorphism of right $R[x,f]$-modules $\D (\iota) : \D(\H) \lra  \D(\ann_{\H} \fB)$ has kernel $\D (\H) \fB$.

\item
Let $\M = (M, \beta)$ be a right $R[x,f]$-module that is Noetherian
as $R$-module. Let $\sigma : \M \fB \lra \M$ denote the inclusion
$R[x,f]$-monomorphism. Then the induced homomorphism of left
$R[x,f]$-modules $\D'(\sigma) : \D'(\M) \lra  \D'(\M \fB)$  has
kernel $\ann_{\D'(\M)} \fB$, so that $\ann_{\D' (\M)} \fB \cong \D'
(\M / \M\fB)$ as left $R[x,f]$-modules.
\end{enumerate}
\end{prop}

\begin{proof}
(i) Since $\fB \subseteq \grann_{R[x,f]}(\ann_{\H} \fB) \subseteq
\grann(\D (\ann_{\H} \fB))_{R[x,f]}$ (by Proposition \ref{ga.2}(i)),
it follows from the fact that $\D(\iota)$ is an
$R[x,f]$-homomorphism that $\D(\H) \fB \subseteq \Ker \D(\iota)$.
There is therefore an induced $R[x,f]$-epimorphism $\phi : \D( \H
)/\D( \H )\fB  \lra   \D (\ann_{\H} \fB)$ for which $\phi( m + \D
(\H) \fB) = \D (\iota) (m)$ for all $m \in \D(\H)$. Let $\lambda :
\D (\H)  \lra  \D(\H)/\D(\H)\fB$  denote the canonical
$R[x,f]$-epimorphism, and note that $\phi \circ \lambda = \D
(\iota)$. We therefore have a commutative diagram
$$
\begin{CD}
\ann_{\H} \fB  @>>{\iota}> \H  @=  \H \\
@VV{ \omega_{\ann_{\H} \fB} }V   @. @V{\cong}V{ \omega_H }V \\
\D'\circ \D (\ann_{\H} \fB)
@>>{ \D'(\phi) }> \D'(\D (\H) /\D(\H) \fB) @>>{\D'(\lambda)}> \D'\circ \D (\H) \\
\end{CD}
$$
in the category $\!\phantom{i}_{R[x,f]}\Mod$. By Proposition
\ref{ga.2}(ii), we have  $\Ima \D'(\lambda) \subseteq \ann_{\D'\circ
\D(\H)}\fB$; since $\omega_H$ is an $R[x,f]$-isomorphism and
$\D'(\lambda)$ and $\D'(\phi)$ are monomorphisms, it follows from
the above commutative diagram that $\D'(\phi)$ is an isomorphism.
Since $I = E_R(R/\fm)$ is an injective cogenerator for $R$, we can
therefore deduce that $\phi$ is an isomorphism, so that $\D(\H) \fB
= \Ker \D(\iota)$.

(ii) Let $j : \ann_{\D'(\M)} \fB \lra \D'(\M)$ denote the inclusion
map and $k : \M \lra \M /\M \fB$ denote the natural epimorphism.
Apply part (i) to the left $R[x,f]$-module $\D'(\M)$ to obtain an
exact sequence
$$
\begin{CD}
0 @>>>  (\D\circ \D' (\M )) \fB @>>> \D\circ \D' (\M ) @>{\D(j)}>> \D(\ann_{\D'(\M)} \fB) @>>> 0 \\
\end{CD}
$$
in $\Mod_{R[x,f]}$. Since $M$ is Noetherian as $R$-module, the
$R[x,f]$-homomorphism $\omega_M : \M  \lra \D'\circ \D(\M)$ is an
isomorphism. There is therefore a commutative diagram
$$
\begin{CD}
0 @>>> \M \fB @>>> \M @>{k}>> \M/\M \fB @>>> 0 \\
@. @V{\cong}VV @V{\cong}V{\omega_M}V @. @. \\
0 @>>> (\D\circ \D' (\M))\fB  @>>> \D\circ \D' (\M) @>{\D(j)}>> \D(\ann_{\D'(\M)} \fB) @>>> 0 \\
\end{CD}
$$
with exact rows in the category $\Mod_{R[x,f]}$.
This induces an $R[x,f]$-isomorphism
$\gamma : \M / \M \fB \stackrel{\cong}{\lra} \D (\ann_{\D'(\M)}\fB)$  which,
when inserted into the above diagram, is such that the extended diagram is still commutative.
Now apply the functor $\D'$ to the right-most square (involving $\gamma$) in that extended diagram:
the result is the right-most square in the commutative diagram
$$
\begin{CD}
\ann_{\D' (\M)} \fB  @>{\omega_{\ann_{\D' (\M)} \fB} }>{\cong}>
\D'\circ \D (\ann_{\D' (\M)} \fB) @>{\D'(\gamma)}>{\cong}> \D' (\M /\M \fB) \\
@V{ \subseteq }V{j}V  @VV{ \D'(\D(j))}V  @VV{\D'(k)}V \\
\D'(\M)  @>{\omega_{M^{\vee}}}>{ \cong }> \D'\circ\D \circ \D' (\M) @>{ \D'(\omega_M)}>{\cong}>  \D'( \M)~~~\qquad. \\
\end{CD}
$$
Note that $\D' (\M)$ is Artinian as $R$-module, so that
$\omega_{M^{\vee}}$ and $\omega_{\ann_{\D'(\M)} (\fB)}$ are both
isomorphisms. Since $\D'(\omega_M) \circ \omega_{M^{\vee}} =
\Id_{M^{\vee}}$ (as noted in Remark \ref{rl.6}), it follows from
this commutative diagram that the kernel of the induced
$R[x,f]$-homomorphism $\D'(\sigma) : \D'(\M) \lra \D'(\M \fB)$,
which is equal to the image of the $R[x,f]$-homomorphism $\D'(k) :
\D'(\M /\M \fB) \lra \D'(\M)$, is precisely $\ann_{\D'(\M)} \fB$.
\end{proof}

\begin{cor}
\label{ga.4} Let the situation and notation be as in\/ {\rm
\ref{main thm}}, and assume in addition that $(R ,\fm)$ is local and
complete and that $I$ is taken to be $E := E_R(R/\fm)$. Let $\M$ be
a right $R[x,f]$-module that is Noetherian as $R$-module. Then $\M$
is $x$-divisible if and only if $\D' (\M)$ is $x$-torsion-free.
\end{cor}

\begin{proof}
Let $j : \M x \lra \M$ be the inclusion $R[x,f]$-monomorphism.
By Proposition \ref{ga.3}(ii), the kernel of $\D'(j) : \D'(\M) \lra \D'(\M x)$ is $\ann_{\D'(\M)} R[x,f]x$.

Now $\M$ is $x$-divisible if and only if $j$ is an isomorphism;
since $E$ is an injective cogenerator for $R$, this is the case if and only if $\D'(j)$ is an isomorphism; and, by the above comment (and the fact that $\D'(j)$ must always be an $R[x,f]$-epimorphism), $\D'(j)$ is an isomorphism if and only if $\ann_{\D'(\M)}R[x,f]x = 0$, that is, if and only if $\D'(\M)$ is $x$-torsion-free.
\end{proof}


\section{\sc Some applications}
\label{appl}

As was mentioned in the Introduction, the
Hartshorne--Speiser--Lyubeznik Theorem has been applied to establish
the existence of a uniform Frobenius test exponent for Frobenius
closures of parameter ideals in a local ring $(R,\fm)$ that is
Cohen--Macaulay, or just generalized Cohen--Macaulay. The non-local
version of the Hartshorne--Speiser--Lyubeznik Theorem given in
\cite[Corollary 1.8]{mmj} can be reformulated as follows: if ($R$ is
not necessarily local and) $H$ is a left $R[x,f]$-module that is
Artinian as $R$-module, then there exists an $e \in \nn$ such that
$\ann_HR[x,f]x^e = \ann_HR[x,f]x^{e+1}$. With this in mind, one can
regard the following result as a `dual
Hartshorne--Speiser--Lyubeznik Theorem' for the case where $(R,\fm)$
is $F$-finite, local and complete.

\begin{thm}
\label{appl.1} Assume that $(R,\fm)$ is $F$-finite, local and
complete. Let $\M = (M, \beta)$ be a right $R[x,f]$-module that is
Noetherian as $R$-module. Then there exists $e \in \nn$ such that
$\M x^e = \M x^{e+1}$, that is, such that $M R[x,f]x^e = M
R[x,f]x^{e+1}$.
\end{thm}

\begin{proof} Let $E := E_R(R/\fm)$.
Select an $(R, R)$-bimodule isomorphism $\Psi : {}_{f}E
\stackrel{\cong}{\lra} \Hom_{rR}(R_f , E)$: recall that Lemma {\rm
\ref{main lemma}} ensures that there is such a $\Psi$. By
Proposition \ref{rt-lt} and Remark \ref{rl.6}, we know that $\D'(\M)
=(M^{\vee}, D(\beta))$ is a left $R[x,f]$-module; by Matlis duality,
as $R$-module, $M^{\vee} = \Hom_R(M,E)$ is Artinian. Therefore, by
the Hartshorne--Speiser--Lyubeznik Theorem, there exists an $e \in
\nn$ such that $\ann_{\D'(\M)} R[x,f]x^e = \ann_{\D'(\M)}
R[x,f]x^{e+1}$.

Let $i : \M x^e \stackrel{\subseteq}{\lra} \M$, $j : \M x^{e+1}
\stackrel{\subseteq}{\lra} \M$ and $k : \M x^{e+1}
\stackrel{\subseteq}{\lra} \M x^e$ be the inclusion
$R[x,f]$-homomorphisms, so that $i\circ k = j$. In view of
Proposition \ref{ga.3}(ii), there is a commutative diagram
$$
\begin{CD}
0 @>>>  \ann_{\D'(\M)} R[x,f]x^e @>>{\subseteq}> \D'(\M) @>>{\D'(i)}> \D'(\M x^e) @>>>  0 \\
@. @V{\subseteq}VV @| @V{\D'(k)}VV  @. \\
0 @>>>  \ann_{\D'(\M)} R[x,f]x^{e+1} @>>{\subseteq}> \D'(\M) @>>{\D'(j)}> \D'(\M x^{e+1}) @>>>  0 \\
\end{CD}
$$
with exact rows in the category $\!\phantom{i}_{R[x,f]}\Mod$.
Since $\ann_{\D'(\M)} R[x,f]x^e = \ann_{\D'(\M)} R[x,f]x^{e+1}$,
we see that $\D'(k)$ must be an isomorphism;
therefore, $k$ must be an isomorphism, because $E$ is an injective cogenerator for $R$.
Therefore $\M x^e = \M x^{e+1}$.
\end{proof}

It is natural to ask whether the conclusion of Theorem \ref{appl.1}
is still valid if we drop the assumptions about $R$ (except the one
that $R$ has characteristic $p$). In Theorem \ref{appl2} below, we
shall show that this is indeed the case. We first present two
preparatory lemmas, in which we assume only that $R$ is a
commutative Noetherian ring of characteristic $p$.

\begin{lem} \label{lem.1}
Let  $\M = (M, \beta)$  be a right $R[x,f]$-module. Suppose that
there is an element  $s \in R$  such that  $Ms \subseteq M x$. Then,
for all $k \in \N$, we have  $Ms^2 \subseteq M x^k$.
\end{lem}

\begin{proof}
We prove the lemma by induction on  $k$. When $k=1$, there is
nothing to prove. Assume that $Ms^2 \subseteq M x^k$  for a $k \in
\N$. Then, since  $Ms^p \subseteq Ms^2 \subseteq M x^k$, we have
$Ms^p x \subseteq Mx ^{k+1}$. Thus $Ms^2 = (Ms)s \subseteq (M x)s =
Ms^p x \subseteq Mx ^{k+1}$.  The lemma is therefore proved by
induction.
\end{proof}

\begin{lem} \label{lem.2}
Let  $\M = (M, \beta)$  be a right $R[x,f]$-module and  $S$ be a
multiplicatively closed subset of  $R$. Then the module of fractions
$S^{-1}M$ has a natural right $(S^{-1}R)[x,f]$-module structure in
which
$$\left(\frac{m}{s}\right)x = \frac{ms^{p-1}x}{s} \quad \text{for
all $m \in M$  and  $s \in S$.}$$ This structure is such that
$S^{-1}(M x^k ) = (S^{-1}M)x^k$ for all $k \in \N$.
\end{lem}

\begin{proof} It is straightforward to construct a
right $(S^{-1}R)[x,f]$-module structure on $S^{-1}M$ with the
specified properties. An easy inductive argument shows that
$$\left(\frac{m}{s}\right)x^k = \frac{ms^{p^k-1}x^k}{s} \quad \text{for
all $k \in \N$, $m \in M$  and  $s \in S$.}$$ It is clear from this
that $S^{-1}(M x^k ) \supseteq (S^{-1}M)x^k$ for a $k \in \N$.

To establish the reverse inclusion, let $\alpha \in S^{-1}(M x^k )$,
so that $\alpha = (mx^k)/s$ for some $m \in M$ and $s \in S$. Then
$$
\alpha = \frac{mx^k}{s} = \frac{mx^ks^{p^k-1}}{s^{p^k}} =
\frac{m(s^{p^k-1})^{p^k}x^k}{s^{p^k}} =
\frac{m(s^{p^k})^{p^k-1}x^k}{s^{p^k}} =
\left(\frac{m}{s^{p^k}}\right)x^k \in (S^{-1}M)x^k.
$$
\end{proof}

\begin{thm}\label{appl2}
Assume only that $R$ is a commutative Noetherian ring of
characteristic $p$. Let $\M = (M,\alpha)$ be a right $R[x,f]$-module
that is Noetherian as $R$-module. Then there exists  $k \in \nn$
such that $\M x^k = \M x^{k+1}$.
\end{thm}

\begin{proof} It is straightforward to check that $(0:_MRx^k) := \{m\in M \ | \ mRx^k = 0\}$ is
an $R[x,f]$-submodule of $M$ for each $k\in \N$. Define $(0 :_{M}
Rx^{\infty}) := \bigcup _{k\in\N} (0:_{M} Rx^k)$; this is also an
$R[x,f]$-submodule of $M$; therefore $M^{\gamma} := M/(0:_M
Rx^{\infty})$ is again a right $R[x,f]$-module.

Since $M$ is Noetherian as $R$-module, the ascending chain
$$
(0:_MRx) \subseteq (0:_MRx^2)\subseteq \cdots \subseteq
(0:_MRx^k)\subseteq \cdots
$$
must eventually be stationary, say at $(0:_MRx^{\ell})$; then $(0
:_{M} Rx^{\infty}) = (0:_MRx^{\ell})$. We point out that if there
exists $k \in \nn$ such that $M ^{\gamma}x^k = M^{\gamma} x^{k+1}$,
then $Mx^k \subseteq Mx^{k+1} + (0:_MRx^{\ell})$, so that
multiplication on the right by $x^{\ell}$ yields that $Mx^{k+\ell}
\subseteq Mx^{k+\ell+1}$. Thus, if the conclusion of the theorem is
true for $M^{\gamma}$, then it is true for $M$.

Next, we define $M \cdot x^{\infty} := \bigcap_{k\in\N}Mx^k$; this
is an $R[x,f]$-submodule of $M$, and so $M^{\sigma} := M/M \cdot
x^{\infty}$ is again a right $R[x,f]$-module. Note also that, if
there exists $k \in \nn$ such that $M ^{\sigma}x^k = M^{\sigma}
x^{k+1}$, then $Mx^k \subseteq Mx^{k+1} + M \cdot x^{\infty} =
Mx^{k+1}$. Thus, if the conclusion of the theorem is true for
$M^{\sigma}$, then it is true for $M$.

Consider the sequence of right $R[x, f]$-modules
$$
M \to M^{\sigma} \to (M^{\sigma})^{\gamma} \to
((M^{\sigma})^{\gamma})^{\sigma} \to
(((M^{\sigma})^{\gamma})^{\sigma})^{\gamma} \to \cdots,
$$
where, at each stage, the arrow denotes the appropriate natural
$R[x,f]$-epimorphism. The first three paragraphs of this proof show
that, if the claim in the theorem holds for any module $M'$ in this
sequence, then it holds for all modules to the left of $M'$,
including $M$ itself. If, for each $n\in\N$, we let $K_n$ denote the
kernel of the composition of the first $n$ epimorphisms in this
sequence, then $K_1 \subseteq K_2 \subseteq \cdots \subseteq K_n
\subseteq \cdots$ is an ascending chain of $R$-submodules of $M$,
and therefore eventually stationary. This means that, in the above
displayed sequence, there is a term to the right of which all the
epimorphisms are isomorphisms. Therefore it is enough for us to
prove the theorem under the additional assumptions that
\begin{equation}\label{assumption}
 (0:_M Rx^{\infty}) =0 \quad \text{and} \quad M \cdot x^{\infty}=0.
\end{equation}

We shall show now that these additional assumptions force $M$ to be
zero (in which case $Mx = Mx^2$). Suppose that $M \not= 0$, and seek
a contradiction. Denote by $\mathrm{Min}_R(M)$  the set of minimal
prime ideals in $\mathrm{Supp}_R(M)$. This is a non-empty set,
because $M \not= 0$. We set  $S := R \setminus \bigcup _{\fp \in
\mathrm{Min}_R(M)} \fp$, a multiplicatively closed subset of $R$.
Then, the $S^{-1}R$-module  $S^{-1}M$  is a non-trivial Artinian
module, since $\dim _{S^{-1}R} S^{-1}M =0$. By Lemma \ref{lem.2},
$S^{-1}M$ has a natural structure as a right
$(S^{-1}R)[x,f]$-module. Consider the descending sequence of right
$(S^{-1}R)[x, f]$-submodules
$$
S^{-1}M \supseteq (S^{-1}M) x \supseteq (S^{-1}M) x^2 \supseteq
(S^{-1}M) x ^3 \supseteq (S^{-1}M) x ^4 \supseteq \cdots
$$
of $S^{-1}M$. Since $S^{-1}M$ is Artinian as $S^{-1}R$-module, we
must have $(S^{-1}M) x^k = (S^{-1}M) x^{k+1}$  for some $k\in\N$.
Then, $S^{-1}(M  x^k) = S^{-1}(M x^{k+1})$ by virtue of Lemma
\ref{lem.2}. Since $M x^k$  is finitely generated as $R$-module, it
follows that there is an element  $s \in S$  such that $(M x^k)s
\subseteq M x^{k+1}$. Now, applying  Lemma \ref{lem.1} to the right
$R[x,f]$-module  $M x^k$,
 we have that  $(M x^k)s^2 \subseteq M x^{k+k'}$  for all $k' \in\N$.
Therefore  $(M x^k)s^2 \subseteq M\cdot x^{\infty} = 0$  by the
assumption  (\ref{assumption}). Since  $(M  x^k) s^2 = M s^{2p^k}
 x^k$, we have shown that  $M s^{2p^k} x^k = 0$; hence
$Ms^{2p^k} \subseteq (0:_M Rx^k) \subseteq (0:_M Rx^{\infty}) =0$.
Consequently we have  $Ms^{2p^k} =0$. Therefore  $s$  belongs to
$\sqrt{\mathrm{ann}_R(M)}$; therefore $s \in \fp$  for all $\fp \in
\mathrm{Min}_R(M)$. But this contradicts the fact that $s$  is an
element of $S$.
\end{proof}


One of the main results of \cite{ga} is that, if $\H$ is an
$x$-torsion-free left $R[x,f]$-module that is Artinian as
$R$-module, then there are only finitely many graded annihilators of
$R[x,f]$-submodules of $\H$. See \cite[Corollary 3.11]{ga}. This
result has relevance to the existence of tight closure test elements
in certain circumstances: see \cite[Corollary 4.7]{ga} and
\cite[Theorem 3.5]{gatcti}. We can use our work in \S1 and \S2 to
obtain a dual result in the special case where $R$ is $F$-finite,
local and complete.

\begin{thm}
\label{appl.3} Assume that $(R,\fm)$ is $F$-finite, local and
complete. Let $\M = (M, \beta)$ be an $x$-divisible right
$R[x,f]$-module that is Noetherian as $R$-module. Then there are
only finitely many graded  annihilators of $R[x,f]$-homomorphic
images of $\M$.
\end{thm}

\begin{proof}
Select an $(R, R)$-bimodule isomorphism $\Psi :{}_fE_R(R/\fm)
\stackrel{\cong}{\lra} \Hom_R(R_f,E_R(R/\fm))$: recall that Lemma
{\rm \ref{main lemma}} ensures that there is such a $\Psi$. Use the
notation of {\rm \ref{main thm}}, but take $I$ to be $E :=
E_R(R/\fm)$. By Proposition  \ref{rt-lt}, we know that $\D'(\M)$ is
a left $R[x,f]$-module; as such, it is $x$-torsion-free, by
Corollary \ref{ga.4}. By Matlis duality, as $R$-module, $M^{\vee} =
\Hom_R(M,E)$ is Artinian. By \cite[Lemma 1.9, Definition 1.10 and
Corollary 3.11]{ga}, there are only finitely many graded
annihilators of $R[x,f]$-submodules of $\D'(\M)$. It is therefore
enough for us to show that, if $\fB$ is the graded-annihilator of
some $R[x,f]$-homomorphic image of $\M$, then $\fB$ is the
graded-annihilator of some $R[x,f]$-submodule of $\D'(\M)$. This we
do.

Thus there is an $R[x,f]$-submodule $\mathbf{L}$ of $\M$ such that
$\fB = \grann(\M /\mathbf{L})_{R[x,f]}$. Therefore $\M \fB \subseteq
\mathbf{L}$, so that there is an $R[x,f]$-epimorphism $\M /\M \fB
\lra \M /\mathbf{L}$. Therefore
$$
\fB \subseteq \grann(\M /\M \fB)_{R[x,f]} \subseteq \grann(\M
/\mathbf{L})_{R[x,f]} = \fB,
$$
so that $\fB = \grann(\M /\M \fB)_{R[x,f]}$.
It now follows from Proposition \ref{ga.2}(iv) that
$$\fB = \grann_{R[x,f]}\D'(\M /\M \fB).$$
However, Proposition \ref{ga.3}(ii) shows that there is an isomorphism
$\D'(\M / \M \fB) \cong \ann_{\D'(\M)} \fB$ of left $R[x,f]$-modules;
therefore $\fB$ is the graded annihilator of the $R[x,f]$-submodule
$\ann_{\D'(\M)}\fB$ of $\D'(\M)$.
This completes the proof.
\end{proof}

It is natural to ask whether the conclusion of Theorem \ref{appl.3}
is still valid if we drop the assumptions about $R$ (except the one
that $R$ has characteristic $p$).

\begin{qu}
\label{appl.4} Assume only that $R$ is (a commutative Noetherian
ring) of characteristic $p$. Let $\M$ be an $x$-divisible right
$R[x,f]$-module that is Noetherian as $R$-module. Is the set of
graded annihilators of $R[x,f]$-homomorphic images of $\M$ finite?
\end{qu}

At the time of writing, we are not able to answer Question
\ref{appl.4}.


\end{document}